\documentclass[letterpaper,10 pt,conference]{ieeeconf}
\IEEEoverridecommandlockouts                              
\usepackage{cite}
\usepackage{amsmath,amssymb,amsfonts}
\usepackage{algorithmic}
\usepackage{graphicx}
\usepackage{textcomp}
\usepackage{xcolor}
\usepackage[utf8]{inputenc}
\usepackage[english]{babel}
\usepackage{fancyhdr}
\usepackage{comment}
\usepackage{times}
\usepackage{changepage}
\usepackage{breqn}
\usepackage{caption}
\usepackage{graphicx}%
\usepackage{epstopdf}
\usepackage{cite}
\usepackage{subfigure}
\usepackage{tikz}
\usepackage[justification=centering]{caption}
\usepackage[utf8]{inputenc}
\usepackage[english]{babel}

\usetikzlibrary{shapes,arrows}
\usetikzlibrary{arrows,shapes,snakes}
\tikzstyle{every picture}+=[remember picture]
\tikzstyle{na} = [baseline=-.5ex]
\usetikzlibrary{shapes.gates.logic.US,trees,positioning,arrows}
\newtheorem{theorem}{Theorem}

\newtheorem{definition}{Definition}
\newtheorem{lemma}{Lemma}

\newtheorem{remark}{\textbf{Remark}}

\newtheorem{assumption}{Assumption}
\DeclareMathOperator*{\argmin}{argmin}
\newcommand{\G}{\mathcal{G}}
\newcommand{\La}{\mathcal{L}}
\newcommand{\R}{\mathbb{R}}
\newcommand{\A}{\mathcal{A}}
\newcommand{\D}{\mathcal{D}}
\newcommand{\E}{\mathcal{E}}
\newcommand{\V}{\mathcal{V}}
\newcommand{\N}{\mathcal{N}}
\newcommand{\sgn}{\mathrm{sgn}}
\newcommand{\SGN}{\mathrm{SGN}}
\newcommand{\1}{\mathbf{1}}
\newcommand{\0}{\mathbf{0}}
\usepackage{hyperref}
\usepackage{cleveref}
\crefname{assumption}{Assumption}{Assumptions}
\def\BibTeX{{\rm B\kern-.05em{\sc i\kern-.025em b}\kern-.08em
    T\kern-.1667em\lower.7ex\hbox{E}\kern-.125emX}}

\title{\LARGE \bf
Distributed Continuous-Time Optimization with Time-Varying Objective Functions and Inequality Constraints
}
\author{{Shan Sun, Wei Ren}
\thanks{This work was supported by National Science Foundation under Grant ECCS-1920798. S. Sun and W. Ren are with the Department of Electrical and Computer Engineering, University of California at Riverside, Riverside, CA, 92521 USA (e-mail: ssun029@ucr.edu, ren@ece.ucr.edu).}}

\begin{document}

\maketitle
\thispagestyle{empty}
\pagestyle{empty}

\begin{abstract}
This paper is devoted to the distributed continuous-time optimization problem with time-varying objective functions and time-varying nonlinear inequality constraints. Different from most studied distributed optimization problems with time-invariant objective functions and constraints, the optimal solution in this paper is time varying and forms a trajectory. To minimize the global time-varying objective function subject to time-varying local constraint functions using only local information and local interaction, we present a distributed control algorithm that consists of a sliding-mode part and a Hessian-based optimization part. The asymptotical convergence of the proposed algorithm to the optimal solution is studied under suitable assumptions. The effectiveness of the proposed scheme is demonstrated through a simulation example. 
\end{abstract}
\section{Introduction}
Distributed optimization algorithms allow for decomposing certain optimization problems into smaller, more manageable sub-problems that can be solved in parallel. Therefore, they are widely used to solve large-scale optimization problems such as optimization of network flows \cite{mainapp2}, big-data analysis \cite{mainapp5}, design of sensor networks \cite{mainapp6}, multi-robot teams \cite{mainapp3}, and resource allocation \cite{mainapp1}. There has been significant attention on distributed convex optimization problems, where the goal is to cooperatively seek the optimal solution that minimizes the sum of private convex objective functions available to each individual agent. In this context, discrete-time distributed optimization algorithms have been studied extensively (see e.g., \cite{DO1,survey} and references therein). 

There exists another body of literature on distributed continuous-time optimization algorithms (see e.g., \cite{DOpeng,DOpeng1,DOren,DOjun3,zhirong,jielu,yiguang}). The distributed continuous-time optimization algorithms have applications in coordinated control of multi-agent teams. For example, multiple physical robots modeled by continuous-time dynamics might need to track a team optimal trajectory. Note that most studies in the literature focus on stationary optimization problems in which both the objective functions and constraints do not explicitly depend on time. However, in many applications, the local performance objectives or engineering constraints may evolve in time, reflecting the fact that the optimal solution could be changing over time and create a trajectory (see e.g., \cite{app4,app3,thesis,thesis1}), which makes the design and analysis much more complex. Moreover in practical optimization problems, constraints are always inevitable. In this paper, we are interested in the distributed continuous-time algorithms for time-varying constrained optimization problems. 

The literature on the distributed continuous-time algorithms for time-varying optimization problems focuses on first-order gradient methods \cite{hu,boxcons} and second-order optimization methods \cite{Rahili}. Specifically, \cite{hu} and \cite{boxcons} solve the distributed time-varying optimization problems with convex set constraints using, respectively, the projected gradient method and gradient-based penalty function method. However, they are limited to solve certain kinds of optimization problems (e.g., \cite{hu}: quadratic objective functions; \cite{boxcons}: linear programming) and there exist tracking errors to the optimal solutions. Although the second-order optimization methods work well in centralized time-varying optimization problems (see e.g., \cite{thesis,thesis1,Fazlyab}), their use in distributed settings has been prohibited as they require global information of the network to compute the inverse of the global Hessian matrix. Ref. \cite{Rahili} solves the distributed time-varying optimization problems using second-order optimization methods. However, the consensus-based algorithm in \cite{Rahili} (Section III.B) is limited to the unconstrained problem with local objective functions that have identical Hessians. While the estimator-based algorithm in \cite{Rahili} (Section III.C) can deal with certain objective functions with nonidentical Hessians, it relies on the distributed average tracking techniques \cite{Chen} and hence poses restrictive assumptions that the derivatives of the Hessians and the gradients of the local objective functions are bounded. In addition, because the estimator-based algorithm has to estimate the Hessian inverse of the global objective function, it necessitates the communication of certain virtual variables between neighbors with increased computation costs. While it is possible to convert the constrained optimization problem to an unconstrained one using penalty methods, the resulting penalized objective functions would not have identical Hessians due to the involvement of the nonuniform
local constraint functions (even if the original objective functions would), and they might not satisfy the restrictive assumptions mentioned above. As a result, the algorithms in \cite{Rahili} cannot be applied to address the distributed time-varying constrained optimization problem (see Remark \ref{remark:comparison} for a more detailed comparison). For discrete-time distributed online optimization algorithms, the readers are referred to \cite{onlineop4} and references therein. 
 
This paper aims to develop a distributed algorithm to solve the continuous-time optimization problem with private time-varying objective functions and time-varying nonlinear inequality constraints. In this work, the time-varying optimization problem is deformed as a consensus subproblem and a minimization subproblem on the team objective function. We develop a sliding-mode method with a Hessian-dependent gain for all the agents to achieve consensus on the states. Meanwhile, a Hessian-based (second-order) optimization method coupled with the log-barrier penalty functions is proposed to track the local time-varying optimal solution. To implement the algorithm, each agent just needs its own state and the relative states between itself and its neighbors. When the agents' states are their positions, the algorithm can be implemented based on purely local sensing (e.g., absolute and relative positions) without the need for communicating virtual variables. The asymptotical convergence to the optimal solution is established based on nonsmooth analysis, Lyapunov theory and convex optimization theory. Numerical simulation results are presented to illustrate the effectiveness of the theoretical results. To the best of our knowledge, this is the first work in the literature on distributed continuous-time time-varying constrained optimization problems that guarantees zero tracking errors. 

\section{PRELIMINARIES}
\subsection{Notation}
Let $\R,\R^n$ and $\R^{n\times m}$ denote the sets of real numbers, real vectors of dimension $n$, and real matrices of size $n\times m$, respectively. Let $\R_{>0}$ represent the set of positive real numbers. The cardinality of a set $S$ is denoted by $|S|$. Let $\1_n$ (resp. $\0_n$) denote the vector of $n$ ones (resp. $n$ zeros), and $I_n$ denote the $n\times n$ identity matrix. For a matrix $A\in \R^{m\times n}$, $[A]_{k\bullet}\in\R^{1\times n}$ is the $k$-th row of $A$, and $A^T$ (resp. $A^{-1}$) is the transpose (resp. inverse) of $A$, $\lambda_{\min}(A)$ is the smallest eigenvalue of $A$. For a vector $x=[x_1,\cdots,x_n]^T\in\R^{n\times 1}$, $\mathrm{diag}(x)\in \R^{n\times n}$ represents the diagonal matrix with the elements in the main diagonal being the elements of $x$, $\|x\|_p$ represents the $p$-norm of the vector $x$, $B(x,\delta)$ represents the open ball of radius $\delta$ centered at $x$, and $\sgn(x)=[\sgn(x_1),\cdots,\sgn(x_n)]^T$, where $\sgn(x_i)$ denotes the signum function defined as 
$$\sgn(x_i)=\left\{\begin{array}{lr}-1&\mathrm{if}\ x_i<0,\\0&\mathrm{if}\ x_i=0,\\1& \mathrm{if}\ x_i>0.\end{array}\right.$$
Let $\nabla f(x, t)$ and $\nabla^2 f(x, t)$ denote, respectively, the gradient and Hessian of function $f(x, t)$ with respect to the vector $x$. Let $\otimes$ denote the Kronecker product and $\overline{\mathrm{co}}$ the convex closure.  

\subsection{Graph Theory}
An undirected graph, is denoted by $\G=(\V,\E,\A)$, where $\V=\{1,...,n\}$ is the node set, $\E\subseteq \V \times \V$ is the edge set, and $\A=[a_{ij}]\in \R^{n\times n}$ is the weighted adjacency matrix with entries $a_{ij},\ i,j\in\V$. For an undirected graph, an edge $(j,i)$ implies that node $i$ and node $j$ are able to share data with each other, and $a_{ij}=1$ if $(j,i) \in \E$ and $a_{ij}=0$ otherwise.  Here $a_{ij}=a_{ji}$. Let $\N_i=\{j\in \V: (j,i)\in \E\}$ denote the set of neighbors of node $i$. A path is a sequence of nodes connected by edges. An undirected graph is connected if for every pair of nodes there is a path connecting them. The Laplacian matrix $\La=[l_{ij}]\in \R^{n\times n}$ associated with $\A$ is defined as $l_{ii}=\sum_{j=1,j\neq i}^na_{ij}$ and $l_{ij}=-a_{ij}$, where $i\neq j$. The incidence matrix $\D=[d_{ij}]\in\R^{n\times |\E|}$ associated with $\G$ is defined as $d_{ik}=-1$ if the $k$th edge leaves node $i$, $d_{ik} = 1$ if it enters node $i$, and $d_{ik} = 0$ otherwise. For the incidence matrix of an undirected graph, the orientation of the edges is assigned arbitrarily. Note that for an undirected graph, $\La\1_n=\0_n$, $\La^T=\La$ and $\La=\D\D^T$.
\subsection{Nonsmooth Analysis}
In this subsection, we recall some important definitions of the nonsmooth systems that will be exploited in our main result.
\begin{definition}\label{definition:filippovsolution} \textit{(Filippov Solution)}\cite{nonsmooth1} Consider the vector differential equation
\begin{equation}\label{filippov}
\dot{x}=f(x,t),
\end{equation}
where $f: \R^d\times \R\rightarrow \R^d$ is measurable and locally essentially bounded. A vector function $x(\cdot)$ is called a Filippov solution of (\ref{filippov}) on $[t_0,t_1]$, if $x(\cdot)$ is absolutely continuous on $[t_0,t_1]$ and for almost all $t\in [t_0,t_1]$, $\dot{x}(t)\in K[f](x,t)$, where
$K[f](x,t):=\bigcap_{\delta>0} \bigcap_{\mu (N)=0} \overline{\mathrm{co}}f\big(B(x,\delta)-N,t\big)$ is the Filippov set-valued map of $f(x,t)$ and $\bigcap_{\mu (N)=0}$ denotes the intersection over all sets $N$ of Lebesgue measure zero.
\end{definition}

\begin{definition}\label{definition:generalizedgradient}\textit{(Clarke's Generalized Gradient)} \cite{nonsmooth1} Consider a locally Lipschitz continuous function $V(x): \R^d\rightarrow \R$, the generalized gradient of the function $V$ at $x$ is given by $
\partial V(x):=\overline{\mathrm{co}}\{\lim \nabla V(x)| x_i\rightarrow x, x_i\not\in\Omega_V \},$
where $\Omega_V$ is the set of measure zero where the gradient of $V$ is not defined. \end{definition}
\begin{definition}\label{deniftion:chainrule}\textit{(Chain Rule)}\cite{nonsmooth1} Let $x(\cdot)$ be a Filippov solution of $\dot{x}=f(x,t)$ and $V(x): \R^d\rightarrow \R$ be a locally Lipschitz continuous function. Then for almost all $t$,
$$
\frac{d}{dt}V[x(t)]\in\dot{\tilde{V}},
$$
where $\dot{\tilde{V}}$ is the set-valued Lie derivative defined as $\dot{\tilde{V}}:=\bigcap_{\xi \in \partial V}\xi^TK[f]$.
\end{definition}
\section{Main Results}
Consider a network consisting of $n$ agents. Each agent is regarded as a node in an undirected graph, and each agent can only interact with its local neighbors in the network. Suppose that each agent satisfies the following continuous-time dynamics
\begin{equation}\label{system}
    \dot{x}_i(t) = u_i(t),
\end{equation}
where $x_i(t)\in\R^m$ is the state of agent $i$, and $u_i(t)\in\R^m$ is the control input of agent $i$. In this work, we study the distributed time-varying optimization problem with time-varying nonlinear inequality constraints. The goal is to design $u_i(t)$ using only local information and interaction, such that all the agents work together to find the optimal trajectory $y^*(t)\in\R^m$ which is defined as
\begin{equation}\label{problem:primal}
\begin{aligned}
&y^*(t) = \argmin  \quad \sum_{i=1}^n f_i[y(t),t],\\
& \text{s.t.} \quad g_i[y(t),t]\leq \0_{q_i},\ i\in\V,
\end{aligned}
\end{equation}
where $f_i[y(t),t]: \R^m \times \R_{>0}\rightarrow \R$ are the local objective functions, and $g_i[y(t),t]: \R^m \times \R_{>0}\rightarrow \R^{q_i}$ are the local inequality constraint functions. It is assumed that $f_i[y(t),t]$ and $g_i[y(t),t]$ are known only to agent $i$. We assume that the local objective functions $f_i[y(t),t]$ and inequality constraint functions $g_i[y(t),t]$ are twice continuously differentiable with respect to $y(t)$ and continuously differentiable with respect $t$. 

If the underlying network is connected, the above problem \eqref{problem:primal} is equivalent to the problem that all the agents reach consensus while optimizing the team objective function $\sum\limits_{i=1}^n f_i[x_i(t),t]$ under constraints, more formally,
\begin{equation}\label{problem}
\begin{aligned}
&x^*(t)\in\R^{m*n} = \argmin  \quad \sum_{i=1}^n f_i[x_i(t),t],\\
& \text{s.t.} \quad g_i[x_i(t),t]\leq \0_{q_i},\quad x_i(t) = x_j(t), \quad \forall i,j\in\V,
\end{aligned}
\end{equation}
where $x(t)\in\R^{m*n}$ is the stack of all the agents$^\prime$ states. Here, the goal is that each state $x_i(t), \ \forall i\in\V$, converges to the optimal solution $y^*(t)$, i.e.,
\begin{equation}
\lim\limits_{t\to\infty}[x_i(t)-y^*(t)]=\0_m.
\end{equation}
This architecture of the distributed time-varying constrained optimization problem with networked agents finds broad applications in distributed cooperative
control problems, including multi-robot navigation \cite{app4,app3} and resource allocation of power network \cite{thesis1}. For notational simplicity, we will remove the time index $t$ from the variables $x_i(t)$ and $u_i(t)$ in most remaining parts of this paper and only keep it in some places when necessary.

\begin{lemma}\label{convex} \cite{Boyd} Let $f(r):R^m\to R$ be a continuously differentiable convex function with respect to $r$. The function $f(r)$ is minimized at $r^*$ if and only if $\nabla f(r^*)=0$.
\end{lemma}

We make the following assumptions which are all standard in the
literature and are used in recent works like \cite{Rahili,Fazlyab}. 
\begin{assumption}
\label{assumption:graph}
The graph $\G$ is fixed, undirected and connected.
\end{assumption}
\begin{assumption}\label{assumption:costfunction}
All the objective functions $f_i(x_i,t)$ are uniformly strongly convex in $x_i$, for all $t\geq 0$.
\end{assumption}

\begin{assumption}\label{assumption:costfunction:ineq}
All the constraints $g_i(x_i,t)$ are uniformly convex in $x_i$, for all $t\geq 0$.
\end{assumption}
\begin{assumption}\label{assumption:slater}
For all $t\geq 0$ and for all $i\in\V$, there exists at least one $y$ such that $g_i(y,t)<\0_{q_i}$. Therefore, the Slater$^\prime$s condition holds for all time.
\end{assumption}

The uniform strong convexity of the objective function implies that the optimal trajectory $y^*(t)$ is unique for all $t \geq 0$. 

\subsection{Distributed Algorithm Design}
In this subsection, we derive our distributed control algorithm for the time-varying constrained optimization problem in \eqref{problem}. 

We design the following controller for agent $i$:
\begin{equation}\label{algorithm}\begin{aligned}
u_i &= -\beta[\nabla^2\tilde{L}_i(x_i,t)]^{-1}\sum_{j\in\N_i} \sgn(x_i-x_j)+\phi_i(t),\\
\phi_i(t)&= -[\nabla^2\tilde{L}_i(x_i,t)]^{-1}\left[\nabla \tilde{L}_i(x_i,t)+\frac{\partial}{\partial t} \nabla \tilde{L}_i(x_i,t)\right],\\
\end{aligned}
\end{equation}
where $\beta\in\R_{>0}$ is a constant control gain, and $\tilde{L}_i(x_i,t)$ is a penalized objective function of agent $i$, defined as, 
\begin{equation}\label{equation:lagrangian}
\begin{aligned}
\tilde{L}_i(x_i,t)&=f_i(x_i,t)-\frac{1}{\rho(t)}\sum_{j=1}^{q_i} \log[1 -\rho(t)g_{ij}(x_i,t)],
\end{aligned}\end{equation}
where $g_{ij}(x_i,t): \R^m \times \R_{>0}\rightarrow \R$ denotes the $j-$th component of function $g_i(x_i,t)$, and $\rho(t)\in \R_{>0}$ is a time-varying barrier parameter satisfying
\begin{equation}\label{gain:rho}
\rho(t)=a_1e^{a_2t},\ a_1,a_2\in\R_{>0}.
\end{equation}
\begin{remark}
In this work, the time-varying optimization problem \eqref{problem} is deformed as a consensus subproblem and a minimization subproblem on the team objective function. We develop a distributed sliding-mode control law to address the consensus part. That is, the role of term $-\beta[\nabla^2\tilde{L}_i(x_i,t)]^{-1}\sum_{j\in\N_i} \sgn(x_i-x_j)$ in \eqref{algorithm} is to drive all the agents to reach a consensus on the states ($\lim\limits_{t\to\infty}\|x_i-\sum_{j=1}^nx_j\|_2=0$). Here, the Hessian-dependent gain $\beta[\nabla^2\tilde{L}_i(x_i,t)]^{-1}$ is introduced to guarantee the convergence of our algorithm under nonidentical $\nabla^2\tilde{L}_i(x_i,t)$. While the second term, $\phi_i(t)\in\R^m$, is an auxiliary variable playing a role in minimizing the penalized objective function $\tilde{L}_i(x_i,t)$ given by \eqref{equation:lagrangian}. Note that we use the log-barrier penalty functions (see the second term in \eqref{equation:lagrangian}) to incorporate the inequality constraints into the penalized objective function. As shown in \eqref{algorithm}, we use the second-order/Hessian information of the penalized objective function to achieve the optimization goal. 
\end{remark}
In addition, we have
\begin{equation}\label{equation:gradient}\begin{aligned}
\nabla\tilde{L}_i(x_i,t)&=\nabla f_i(x_i,t)+\sum_{j=1}^{q_i} \frac{\nabla g_{ij}(x_i,t)}{1-\rho(t) g_{ij}(x_i,t)},\\
\end{aligned}
\end{equation}
\begin{equation}\label{equation:partialgradient}\begin{aligned}
\frac{\partial}{\partial t}\nabla\tilde{L}_i(x_i,t)&=\frac{\partial }{\partial t}\nabla f_i(x_i,t)+\sum_{j=1}^{q_i} \frac{\partial \nabla g_{ij}(x_i,t)/\partial t}{1-\rho(t)g_{ij}(x_i,t)}\\
&\qquad+\sum_{j=1}^{q_i} \frac{\dot{\rho}(t)g_{ij}(x_i,t)\nabla g_{ij}(x_i,t)}{[1-\rho(t)g_{ij}(x_i,t)]^2},\\
&\qquad+\sum_{j=1}^{q_i}\frac{\rho(t)\nabla g_{ij}(x_i,t)\partial g_{ij}(x_i,t)/\partial t}{[1-\rho(t)g_{ij}(x_i,t)]^2},\\
\end{aligned}\end{equation}
\begin{equation}\label{equation:hessian}\begin{aligned}
\nabla^2\tilde{L}_i(x_i,t)&=\nabla^2f_i(x_i,t)+\sum_{j=1}^{q_i} \frac{\nabla^2 g_{ij}(x_i,t)}{1-\rho(t)g_{ij}(x_i,t)}\\
&\qquad+\sum_{j=1}^{q_i} \frac{\rho(t)\nabla g_{ij}(x_i,t)\nabla g_{ij}(x_i,t)^T}{[1-\rho(t)g_{ij}(x_i,t)]^2},\\
\end{aligned}\end{equation}
where $\frac{\partial}{\partial t}\nabla f_i(x_i,t)$, $\frac{\partial}{\partial t}\nabla g_{ij}(x_i,t)$ and $\frac{\partial}{\partial t}g_{ij}(x_i,t)$ are, respectively, the partial derivatives of $\nabla f_i(x_i,t)$, $\nabla g_{ij}(x_i,t)$ and $g_{ij}(x_i,t)$ with respect to $t$. To make the algorithm \eqref{algorithm} work, the initial states $x_i(0)$ need satisfy

\begin{equation}\label{ini:x}
g_{i}[x_i(0),0]< \0_{q_i}, \qquad \forall i\in\V.
\end{equation}
Also, for notational simplicity, we will remove the time index $t$ from the auxiliary variable $\phi_i(t)$ in most remaining parts of this paper and only keep it in some places when necessary.
\begin{remark}\label{remark:comparison}
In this work, we convert the considered constrained optimization problem into an unconstrained one using the log-barrier penalty functions. It is worth noting that the proposed algorithm \eqref{algorithm} is not a simple extension of the existing distributed time-varying unconstrained optimization algorithms in \cite{Rahili}. Especially, to apply the consensus-based algorithm in \cite{Rahili} (Section III.B), it is required that the Hessians of all the local objective functions be identical. In contrast, in our context with the penalized objective functions, the Hessians of them are nonuniform
due to the involvement of the nonuniform local constraint functions even if the original objective functions have identical Hessians. The estimator-based algorithm in \cite{Rahili} (Section III.C) can deal with certain objective functions with nonidentical Hessians. However, it not only necessitates the communication
of certain virtual variables between neighbors with increased computation costs, but requires that the derivatives of the Hessians and the gradients of the objective functions be bounded. Unfortunately, due to the complexity of the penalized objective functions in the considered constrained problem, such a requirement is no longer guaranteed to hold and hence the result therein is not applicable to our problem. In this paper, we introduce a novel algorithm with a Hessian-dependent gain to account for the complexity caused by the penalized objective functions. The novel algorithm design in turn introduces new challenges in theoretical analysis, which will be addressed in the following. 
\end{remark}
\begin{remark}
In algorithm \eqref{algorithm}, each agent just needs its own information and the relative states between itself and its neighbors. In some robotic applications, the agents' states are their spatial positions. As a result, the relative positions can be obtained by local sensing and the communication necessity might be eliminated. 
\end{remark}
\subsection{Convergence Analysis}
In this subsection, the asymptotical convergence of system \eqref{system} under controller \eqref{algorithm} to the optimal solution is established. To establish our results, we require the following assumptions.
\begin{assumption}\label{assumption:bound}
If all the local states $x_i$ are bounded, then there exists a constant $\bar{\alpha}$ such that $\sup\limits_{t\in[0,\infty)}\|\frac{\partial }{\partial t} \nabla f_i(x_i,t)\|_{2} \leq \bar{\alpha}$, for all $i\in\V$ and $t\geq 0$.
\end{assumption}
\begin{assumption}\label{assumption:bound:ineq}
If all the local states $x_i$ are bounded, then there exist constants $\bar{\beta}$ and $\bar{\gamma}$ such that $\sup\limits_{t\in[0,\infty)}\|\frac{\partial }{\partial t} \nabla g_{ij}(x_i,t)\|_{2} \leq \bar{\beta}$ and $\sup\limits_{t\in[0,\infty)}\|\frac{\partial }{\partial t} g_{ij}(x_i,t)\|_{2} \leq \bar{\gamma}$, for all $ i\in\V,\ j\in[1,\cdots,q_i]$ and $t\geq 0$. 
\end{assumption}
\begin{remark}
In Assumption \ref{assumption:bound}, we assume that all $\|\frac{\partial}{\partial t}\nabla f_i(x_i,t)\|_2$ are bounded under bounded $x_i$. The assumption holds for an important class of situations. For example, consider the normal quadratic objective functions $\|c_ix_i+h_i(x_i,t)\|^2_2$. As long as $\frac{\partial}{\partial t}h_i(x_i,t)$ (e.g. $\sin(t), t$) are bounded under bounded $x_i$,  $\|\frac{\partial}{\partial t}\nabla f_i(x_i,t)\|_2$ will be bounded. In Assumption \ref{assumption:bound:ineq}, we assume that all $\|\frac{\partial}{\partial t}\nabla g_{ij}(x_i,t)\|_2$ and $\|\frac{\partial}{\partial t}g_{ij}(x_i,t)\|_2$ are bounded under bounded $x_i$. The assumption holds for an important class of situations. The boundedness of $\|\frac{\partial}{\partial t}\nabla g_{ij}(x_i,t)\|_2$ and $\|\frac{\partial}{\partial t}g_{ij}(x_i,t)\|_2$ holds for most commonly used boundary constraint functions, e.g., $x_i\leq b(t)$ or $x^2_i\leq b(t)$ under bounded $\dot{b}(t)$. 
\end{remark}
\begin{remark}\label{remark:nonsmooth}
With the piecewise-differentiable signum function involved in algorithm \eqref{algorithm}, the solution should be investigated in the sense of Filippov \cite{cortes}. However, since the signum function is measurable and locally essentially bounded, the Filippov solutions of the proposed system dynamics always exist. Hence if the Lyapunov function candidates are continuously differentiable, the set-valued Lie derivative of them is a singleton at the discontinuous points and the proof still holds without employing the nonsmooth analysis. 
\end{remark}
In this work, we convert the considered constrained optimization problem into an unconstrained optimization problem using the log-barrier penalty functions. It is important that the log-barrier penalty function involved in \eqref{equation:lagrangian} is always well defined under our proposed algorithm. This is described in the next lemma.

\begin{lemma}\label{lemma:convex}
Suppose that \cref{assumption:slater} and the initial condition \eqref{ini:x} hold. For the system \eqref{system} under controller \eqref{algorithm}, each $x_i(t)$ belongs to set $D_i=\{x_i\in\R^m\ |\ g_i(x_i,t)< \frac{1}{\rho(t)}\1_{q_i}\}$ for all $t\geq 0$. That is, \eqref{equation:lagrangian} is always well defined. 
\end{lemma}
\begin{proof}
See Appendix \ref{appendix:convex}.
\end{proof}
In the following, in Lemma \ref{lemma:sumgradient}, we prove that the eventual states of the agents satisfy the optimal requirement shown in Lemma \ref{convex}, i.e., $\lim\limits_{t\to\infty}\sum\limits_{i=1}^n\nabla \tilde{L}_i[x_i(t),t] =0$. The goal of problem \eqref{problem} is that all the agents' states reach consensus on the optimal trajectory, and thus in Lemma \ref{lemma:consensus}, we prove that consensus can be achieved eventually if all $\phi_i$ in the controller \eqref{algorithm} are bounded, i.e., $\lim\limits_{t\to\infty}\|x_i(t)-\sum\limits_{j=1}^n x_j(t)\|=0$ can be achieved under bounded $\phi_i$. Then in Lemma \ref{lemma:phi:ineq}, we prove that all $\phi_i$ associated with the system \eqref{system} under controller \eqref{algorithm} are indeed bounded. Finally, in Theorem \ref{theorem}, we present that the original constrained optimization problem \eqref{problem} can be achieved, i.e., $\lim\limits_{t\to\infty}\|x_i(t)-y^*(t)\|_2 = 0$ for all $i\in\V$.
\begin{lemma}\label{lemma:sumgradient} 
Suppose that \cref{assumption:graph,assumption:costfunction,assumption:costfunction:ineq,assumption:slater}, the gain condition \eqref{gain:rho} and initial condition \eqref{ini:x} hold. For the system \eqref{system} under controller \eqref{algorithm}, the summation of all $\nabla \tilde{L}_i(x_i,t)$ satisfies $\lim\limits_{t\to\infty}\sum_{i=1}^n \nabla \tilde{L}_i(x_i,t)=\0_m$. 
\end{lemma}
\begin{proof}
See Appendix \ref{appendix:sumgradient}.
\end{proof}
\begin{lemma}\label{lemma:consensus}
Suppose that \cref{assumption:graph,assumption:costfunction,assumption:costfunction:ineq,assumption:slater}, the gain condition \eqref{gain:rho} and initial condition \eqref{ini:x} hold. For the system \eqref{system} under controller \eqref{algorithm}, if there exists a constant $\bar{\phi}$ such that $\sup\limits_{t\in[0,\infty)} \|\phi_i(t)\|_2\leq \bar{\phi},\ \forall i\in\V$ and $\beta$ satisfies that,
\begin{equation}\label{beta}
    \beta> \frac{2\bar{\phi}mn^2|\E|}{\min\{\lambda_{\min}[(\nabla^2\tilde{L}_i)^{-1}]\}},
\end{equation}
all the states $x_i$ will achieve consensus eventually, i.e., $\lim\limits_{t\to\infty}\|x_i(t)-x_j(t)\|_2=0, \ \forall i,j\in\V$.
\end{lemma}
\begin{proof}
See Appendix \ref{appendix:consensus}.
\end{proof}
\begin{lemma}\label{lemma:phi:ineq}
Suppose that \cref{assumption:graph,assumption:costfunction,assumption:bound,assumption:costfunction:ineq,assumption:slater,assumption:bound:ineq}, the gain condition \eqref{gain:rho} and initial condition \eqref{ini:x} hold. For the system \eqref{system} under controller \eqref{algorithm}, all $\phi_i$ remain bounded. That is, there exists a constant $\bar{\phi}$ such that $\sup_{t\in[0,\infty)}\|\phi_i(t)\|_2\leq \bar{\phi}, \ \forall i\in\V$.
\end{lemma}
\begin{proof}
See Appendix \ref{appendix:phi}.
\end{proof}
\begin{theorem}\label{theorem}
Suppose that \cref{assumption:graph,assumption:costfunction,assumption:bound,assumption:costfunction:ineq,assumption:bound:ineq,assumption:slater}, the initial condition \eqref{ini:x} and gain conditions \eqref{gain:rho} and \eqref{beta} hold. For the system \eqref{system} under controller \eqref{algorithm}, all the states $x_i$ will converge to the optimal solution $y^*(t)$ in \eqref{problem:primal} eventually.
\end{theorem}
\begin{proof}
Define 
\begin{equation}\label{ystar:ineq}
\tilde{y}(t)^*\in\R^{m}=\text{argmin}\sum_{i=1}^n \tilde{L}_i[y(t),t],
\end{equation}
where $\tilde{L}_i[y(t),t]$ is each agent's penalized objective function defined by \eqref{equation:lagrangian}. Note that \cref{assumption:graph,assumption:costfunction,assumption:bound,assumption:costfunction:ineq,assumption:slater,assumption:bound:ineq}, the initial condition \eqref{ini:x} and gain condition \eqref{gain:rho} hold. It follows from Lemma \ref{lemma:phi:ineq} that all $\phi_i$ associated with the system \eqref{system} under controller \eqref{algorithm} are bounded for all $t\geq 0$, which in turn implies that $x_i(t)=x_j(t),\ \forall i,j\in\V$ eventually according to Lemma \ref{lemma:consensus}. Moreover, based on Lemma \ref{lemma:sumgradient} we know that $\lim\limits_{t\to\infty} \sum\limits_{i=1}^n\nabla \tilde{L}_i(x_i,t)=\0_m$. Using a similar analysis to that in Lemma \ref{lemma:sumgradient}, we have each $\tilde{L}_i(x_i,t)$ is continuously differentiable and strongly convex in $x_i$. Then it follows from Lemma \ref{convex} that all $x_i$ will converge to the optimal solution $\tilde{y}^*(t)$ in \eqref{ystar:ineq}, i.e., $\lim\limits_{t\to\infty} x_i(t)=\tilde{y}^*(t),\ \forall i\in\V$.
		
Define 
$$
\begin{aligned}
&\hat{y}^*(t)\in\R^m=\text{argmin}\sum_{i=1}^n f_i[y(t),t]\\
&\text{s.t.}\quad g_{ij}[y(t),t]\leq \frac{1}{\rho(t)},\ \forall i\in\V, j=1,\cdots,q_i. 
\end{aligned}
$$
According to \cite{Boyd} (Sec. 11.2), we know that
$$
\left|\sum\limits_{i=1}^nf_i[\hat{y}^*(t),t]-\sum\limits_{i=1}^nf_i[\tilde{y}^*(t),t]\right| \leq \sum\limits_{j=1}^n\frac{q_j}{\rho(t)}.
$$
Note that $y^*(t)\in\R^m$ is the optimal solution of problem \eqref{problem:primal}. Under Assumption \ref{assumption:slater}, the optimal solution $y^*(t)$ can be characterized using the Karush–Kuhn–Tucker (KKT) conditions for all $t\geq 0$. Then based on \cite{Boyd} (Sec. 5.9), we have
$$
\left|\sum\limits_{i=1}^nf_i[\hat{y}^*(t),t]-\sum\limits_{i=1}^nf_i[y^*(t),t]\right| \leq \sum_{j=1}^n \sum_{k=1}^{q_j} \frac{\nu^*_{jk}(t) }{\rho(t)},
$$
where $\nu_{jk}(t)$ are the Lagrangian multipliers corresponding to the inequality constraint defined in \eqref{problem:primal}, and $\nu_{jk}^*(t)$ are the optimal Lagrangian multipliers.  
Hence, because $\lim_{t\rightarrow \infty}\rho(t) = \infty$, we have
$$
\lim_{t\rightarrow\infty}\left|\sum\limits_{i=1}^nf_i[y^*(t),t]-\sum\limits_{i=1}^nf_i[\tilde{y}^*(t),t]\right| = 0.
$$
Under Assumption \ref{assumption:costfunction}, the optimal solution $y^*(t)$ is unique, which indicates that $\lim\limits_{t\to\infty} x_i(t)=y^*(t),\ \forall i\in\V$.
\end{proof}

\section{NUMERICAL SIMULATION RESULTS}
\begin{figure}[!htb]
		\centering
		\begin{tikzpicture}[<->,>=stealth',shorten >=1pt,auto,node distance=1.1cm,
		thick,main node/.style={circle,fill=blue!20,draw,font=\sffamily\tiny\bfseries,scale=0.85}]
		\node[main node] (1) {1};
		\node[main node] (2) [right of=1] {2};
		\node[main node] (3) [right of=2] {3};
		\node[main node] (4) [right of=3] {4};
		\node[main node] (5) [right of=4] {5};	
		\node[main node] (6) [below of=2] {6};
		\node[main node] (7) [below of=4] {7};
		\node[main node] (9) [below of=6] {9};
		\node[main node] (10) [right of=9] {10};
		\node[main node] (8) [below of=1] {8};
		\node[main node] (11) [below of=7] {11};
		\node[main node] (12) [below of=5] {12};
		\path[every node/.style={font=\sffamily\small}]
		(1) edge  node   {} (2)
		(2) edge  node   {} (3)
		(1) edge  node   {} (8)
		(5) edge  node   {} (12)
		(3) edge  node   {} (4)
		(4) edge  node   {} (5)
	    (3) edge  node   {} (6)
	    (3) edge  node   {} (7)
		(6) edge  node   {} (10)
	    (7) edge  node   {} (10)
		(8) edge  node   {} (9)
        (9) edge  node   {} (10)
        (10) edge  node {} (11)
        (11) edge  node {} (12);
		\end{tikzpicture}
		\caption{An undirected graph. \label{undirectedgraph}}
\end{figure}
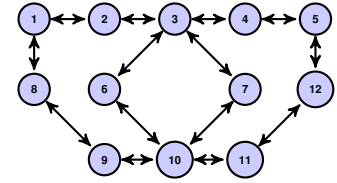
\begin{figure}[!htb]
		\centering
		\includegraphics[width=0.8\linewidth, height=0.23\textheight]{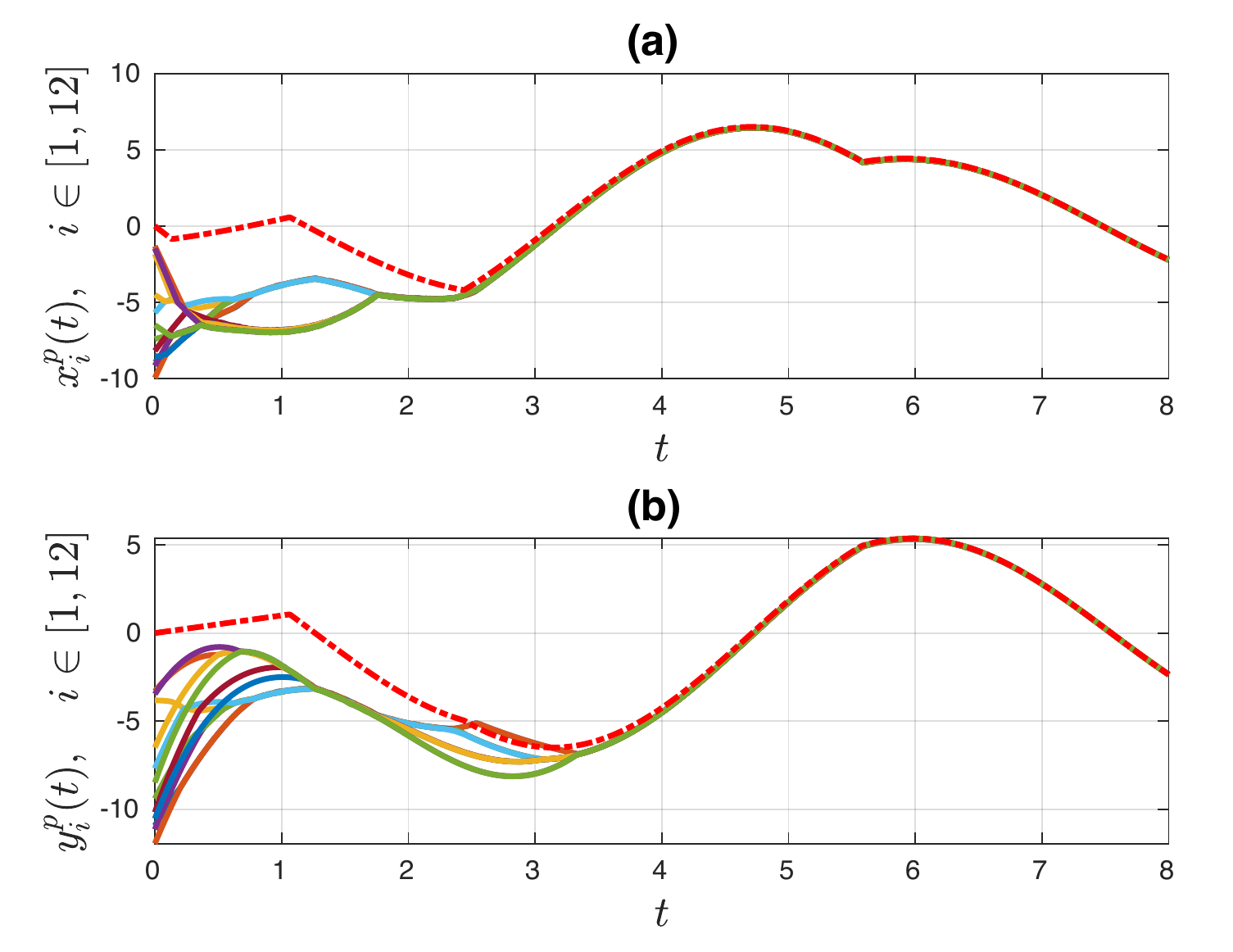}
		\caption{State trajectories of all the agents corresponding to Theorem \ref{theorem}. The red dashed line is the optimal solution and the other solid lines are the trajectories of all agents' states.}
		\label{result_trajectory}
\end{figure}
\begin{figure}[!htb]
	\centering
	\includegraphics[width=0.8\linewidth, height=0.23\textheight]{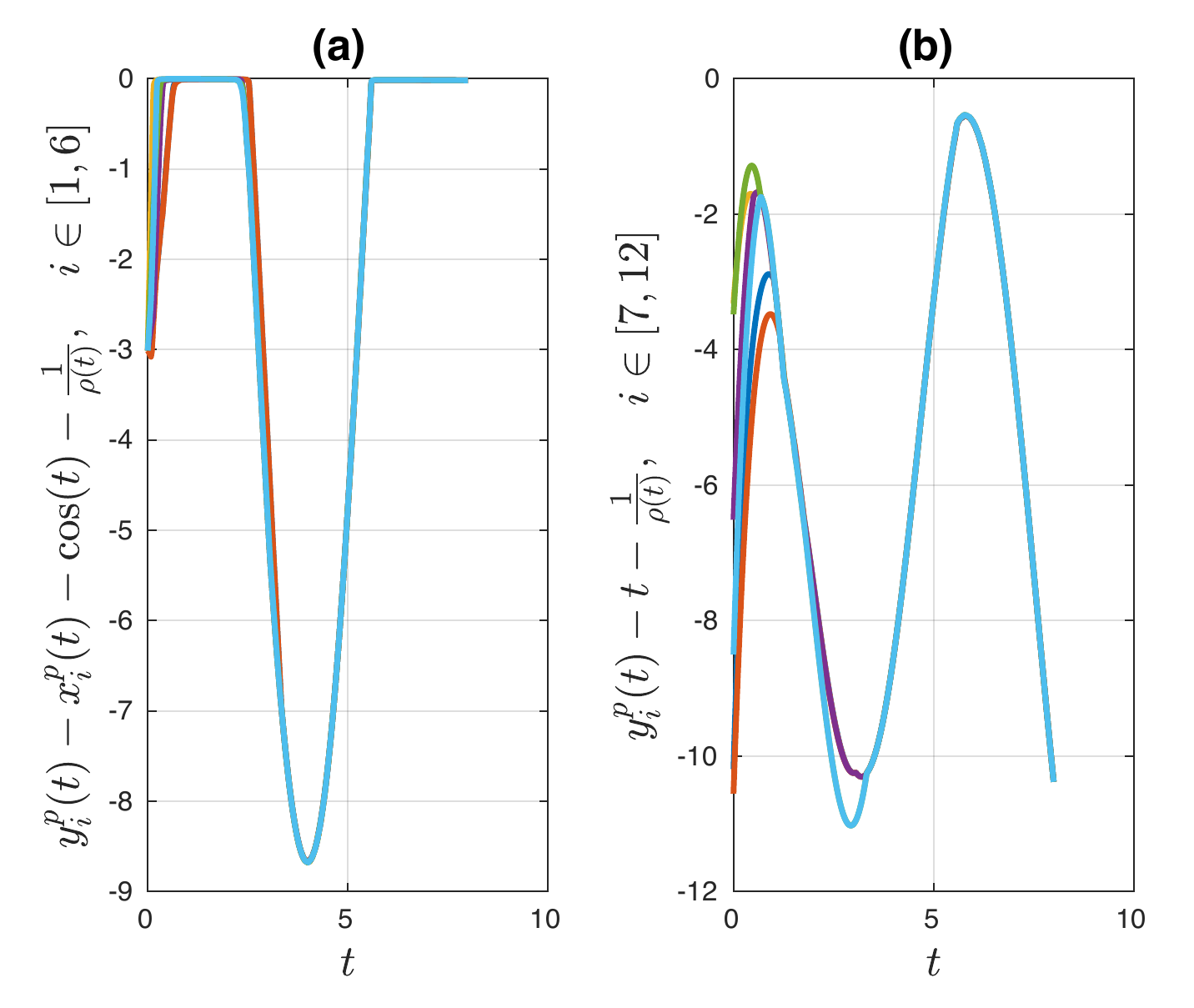}
	\caption{(a) The constraint results of agents $1-6$.\\ $\qquad\quad\ $ (b) The constraint results of agents $7-12$.}
	\label{result_constraint}
\end{figure}
In this section, we use a simulation case to illustrate Theorem \ref{theorem}. We consider a network with $n=12$ and $m=2$. The network topology is shown by the undirected graph in Figure \ref{undirectedgraph}. Let $x_i=[x_i^p,y^p_i]^T\in\R^2$ denote the states of each agent and consider the following optimization problem:
\begin{equation}\begin{aligned}
    \min &\quad\sum\limits_{i=1}^{12} \frac{1}{2}[x^p_i(t)+i\sin(t)]^2+\frac{3}{2}[y^p_i(t)-i\cos(t)]^2,\\
    \text{s.t.} &\quad y^p_j(t)-x^p_j(t)-\cos(t)\leq0,\ \forall j\in[1,2,\cdots,6],\\
    &\quad y^p_k(t)-t\leq0,\ \forall k\in[7,8,\cdots,12],\\
    &\quad x_i^p = x_j^p\  \text{and}\  y_i^p = y_j^p,\ \forall i,j\in\V.
\end{aligned}\end{equation}

The initial states $x^p_i(0),\ i\in\V$ are generated randomly from the range $[-10,0]$, and $y_i^p(0)=x_i^p(0)-2,\ i\in\V$. Therefore, the initial condition \eqref{ini:x} is satisfied. We choose $\beta=25$, $\rho(t)=100\exp(0.1t)$. Therefore, the gain condition \eqref{gain:rho} is satisfied. The state trajectories of all the agents are shown in Figure \ref{result_trajectory}. We can see that all the agents are able to track the optimal trajectory without tracking errors eventually which is consistent with Theorem \ref{theorem}. The constraint result is shown in Figure \ref{result_constraint}. In our simulation, agents $1-6$ are assigned to the constraint function $y^p_i(t)-x^p_i(t)-\cos(t)\leq 0, \ i\in[1,\cdots,6]$ , so $y^p_i(t)-x^p_i(t)-\cos(t)-1/\rho(t), \ i\in[1,\cdots,6]$ always remain negative. Agents $7-12$ are assigned to the constraint function $y^p_i(t)-t\leq 0, \ i\in[7,\cdots,12]$ , and thus $y^p_i(t)-t-1/\rho(t), \ i\in[7,\cdots,12]$ always remain negative. \\
\section{CONCLUSIONS}

In this paper, we have studied the distributed continuous-time constrained optimization problem with time-varying objective functions and time-varying constraints. The goal is that a set of networked agents cooperate to track the time-varying optimal solution that minimizes the summation of all the local time-varying objective functions subject to all the local time-varying constraints, where each agent can only receive information from its neighbors. We propose a distributed sliding-mode algorithm built on the Hessian-based optimization methodology. We have shown that asymptotical convergence is guaranteed under some reasonable assumptions. Numerical simulation result is given to illustrate the theoretical algorithm.

\begin{appendices}
\section{Proof of Lemma \ref{lemma:convex}}\label{appendix:convex}
Assumption \ref{assumption:slater} ensures the existence of the initial condition \eqref{ini:x}. Moreover, the derivative of $\nabla \tilde{L}_i(x_i,t)$ is given by
\begin{equation}\label{a1}
\dot{\nabla} \tilde{L}_i(x_i,t)=\nabla^2\tilde{L}_i(x_i,t) \dot{x}_i +\frac{\partial}{\partial t}\nabla \tilde{L}_i(x_i,t).
\end{equation}
Substituting the solution of \eqref{system} with \eqref{algorithm} into \eqref{a1} leads to
\begin{equation}\begin{aligned}\label{dotgradient}
\dot{\nabla} \tilde{L}_i(x_i,t)&=-\beta\sum\limits_{j\in\N_i}\SGN(x_i-x_j)-\nabla \tilde{L}_i(x_i,t),
\end{aligned}
\end{equation}  
where $\SGN(\cdot)$ \footnote{With the piecewise-differentiable signum function involved in algorithm \eqref{algorithm}, the solution of \eqref{system} with \eqref{algorithm} should be replaced by inclusions at a point of discontinuity. It follows from \cite{cortes} that the $\text{SGN}$ function in Equation \eqref{SGN} is the Filippov set-valued map of the signum function which considers inclusions at the discontinuity point $z = 0$ .} is the multivalued function defined as
\begin{equation}\label{SGN}
\SGN(z)=\left\{\begin{array}{ll}1&\text{if} \ z > 0,\\
\left[-1,1\right]&\text{if} \ z = 0,\\
-1&\text{if} \ z < 0.\end{array} \right.
\end{equation} 
It is obvious that each $\nabla\tilde{L}_i(x_i,t)$ remains bounded for all $t\geq 0$. Notice that \eqref{equation:gradient} implies that $\nabla\tilde{L}_i(x_i,t)$ is unbouned at the boundary of $D_i$. Therefore, it follows from the initial condition \eqref{ini:x} that each $x_i$ is in the set $D_i=\{x_i\in\R^m\ |\ g_i(x_i,t)< \frac{1}{\rho(t)}\1_{q_i}\}$ for all $t\geq 0$. That is, \eqref{equation:lagrangian} is always well defined.

\section{Proof of Lemma \ref{lemma:sumgradient}} \label{appendix:sumgradient}
It follows from Assumption \ref{assumption:costfunction} that all $f_i(x_i,t)$ are strongly convex in $x_i$. Also it follows from Assumption \ref{assumption:costfunction:ineq} that all $g_{ij}(x_i,t)$ are convex in $x_i$. From the gain condition \eqref{gain:rho}, we know that $\rho(t)$ is always positive. Then it follows from the initial condition \eqref{ini:x} that $\tilde{L}_i(x_i,t)$ given by \eqref{equation:lagrangian} must be continuously differentiable and strongly convex in $x_i$ if $x_i$ is in the set $D_i=\{x_i\in\R^m\ |\ g_i(x_i,t)< \frac{1}{\rho(t)}\1_{q_i}\}$. Note that \cref{assumption:slater} and the initial condition \eqref{ini:x} hold. Lemma \ref{lemma:convex} has indicated that each $x_i$ is indeed the case. Therefore, each $\tilde{L}_i(x_i,t)$ must be continuously differentiable and strongly convex in $x_i$ based on our algorithm. 
Consider the Lyapunov function candidate,
\begin{equation}
W_1=\frac{1}{2}\left[\sum_{i=1}^n \nabla \tilde{L}_i(x_i,t)\right]^T\left[\sum_{i=1}^n \nabla \tilde{L}_i(x_i,t)\right].
\end{equation}
Note that the Lyapunov candidate $W_1$ is continuously differentiable. Based on the statements in Remark \ref{remark:nonsmooth}, we do not need to employ nonsmooth analysis in the stability analysis. 
Then we have
\begin{equation}\label{a2}
\begin{aligned}
\dot{W}_1(t) &=\left[\sum_{i=1}^n \nabla \tilde{L}_i(x_i,t)\right]^T\\
&\qquad\times\left[\sum_{i=1}^n \nabla^2\tilde{L}_i(x_i,t)\dot{x}_i+\frac{\partial}{\partial t}\nabla \tilde{L}_i(x_i,t)\right].
\end{aligned}\end{equation}
Substituting the solution of \eqref{system} with \eqref{algorithm} into \eqref{a2} leads to
$$\begin{aligned}
\dot{W}_1(t) &=\left[\sum_{i=1}^n \nabla \tilde{L}_i(x_i,t)\right]^T\left(\sum_{i=1}^n \nabla^2\tilde{L}_i(x_i,t)\right.\\
&\times\left\{[\nabla^2\tilde{L}_i(x_i,t)]^{-1}
\beta\sum_{j\in\N_i}\sgn(x_j-x_i)+\phi_i\right\}\\
&\qquad+\left.\frac{\partial}{\partial t}\nabla \tilde{L}_i(x_i,t)\right).
\end{aligned}
$$
Since the network is undirected (Assumption \ref{assumption:graph}), we have $\sum\limits_{i=1}^n\beta\sum\limits_{j\in\N_i}\sgn(x_j-x_i)=\0_m$ for all $t\geq 0$. It follows that
$$
\dot{W}_1(t) =\left[\sum_{i=1}^n \nabla \tilde{L}_i(x_i,t)\right]^T\left[-\sum_{i=1}^n\nabla \tilde{L}_i(x_i,t)\right]=-2 W_1(t).	
$$
indicating that $W_1(t)=e^{-2t}W_1(0)$ for all $t\geq 0$. It can be concluded that $\lim\limits_{t\to\infty} W_1(t)=0$, and thus $\lim\limits_{t\to\infty}\sum_{i=1}^n \nabla \tilde{L}_i(x_i,t)=\0_m$.


\section{Proof of Lemma \ref{lemma:consensus}}\label{appendix:consensus}
Define
$$
\begin{aligned}
\left[\nabla^2\tilde{L} (x,t)\right]^{-1}&=\text{diag}\left\{[\nabla^2\tilde{L}_1(x_1,t)]^{-1},\right.\\
&\qquad\left.\cdots,[\nabla^2\tilde{L}_n(x_n,t)]^{-1}\right\},\\
x&=[x_1^T,\cdots,x_n^T]^T,\\
\Phi&=[\phi_1^T,\cdots, \phi_n^T]^T.
\end{aligned}$$
Consider the Lyapunov candidate 
\begin{equation}
W_2(t)=\| (\D^T \otimes I_{m} )x\|_1.
\end{equation}
The solution of \eqref{system} with \eqref{algorithm} can be written in compact form as
\begin{equation}\label{equation:dynamical}
\dot{x}=-\beta [\nabla^2\tilde{L}(x,t)]^{-1}(\D\otimes I_m)\sgn[(\D^T\otimes I_m) x]+\Phi.
\end{equation}
It is obvious that $W_2(t)$ is locally Lipschitz continuous but nonsmooth at some points. Then according to Definition \ref{definition:generalizedgradient}, the generalized gradient of $W_2(t)$ is given by
\begin{equation}
\partial W_2(t) =  (\D^T \otimes I_{m} ) ^T\{\SGN[(\D^T\otimes I_{m})x]\},
\end{equation}
where $\SGN(\cdot)$ is defined in \eqref{SGN}. Then based on Definition \ref{deniftion:chainrule}, the set-valued Lie derivative of $W_2(t)$ is given by
\begin{equation}\label{dotw2}\begin{aligned}
\dot{\tilde{W}}_2(t) &=\bigcap_{\xi \in \SGN[(\D^T\otimes I_m)x]}\xi^T(\D^T\otimes I_m)K[f],
\end{aligned}\end{equation}
where $K[f]=\Phi-\beta [\nabla^2\tilde{L}(x,t)]^{-1} (\D \otimes I_{m})\SGN [(\D^T\otimes I_{m})x]$ is the set-valued Filippov map of the dynamical system \eqref{equation:dynamical}.

Since there is an intersection operation on the right side of \eqref{dotw2}, it follows that as long as $\dot{\tilde{W}}_2(t)$ is not empty and there exists $\xi \in \SGN[(\D^T\otimes I_m)x]$ such that $\xi^T(\D^T\otimes I_m)\tilde{f}<0,\ \forall \tilde{f}\in K[f]$, then the result of $\dot{\tilde{W}}_2(t)$ falls into the negative half plane of the real axis. Arbitrarily choose $\eta\in\SGN[(\D^T\otimes I_{m})x]$. Choose $\xi_k=\sgn[(\D^T\otimes I_m)_{k\bullet}x]$ if $\sgn[(\D^T\otimes I_m)_{k\bullet}x]\neq 0$ and $\xi_k=\eta_k$ if $\sgn[(\D^T\otimes I_m)_{k\bullet}x]= 0$, where $\xi_k$ and $\eta_k$ denote the $k$th element in vectors $\xi$ and $\eta$ respectively. If $\dot{\tilde{W}}_2(t)\neq \emptyset$, suppose that $\tilde{a}\in\dot{\tilde{W}}_2(t)$. It follows that
\begin{equation}\begin{aligned}
\tilde{a} &=-\beta \{\xi^T (\D^T \otimes I_{m} )[\nabla^2\tilde{L}(x,t)]^{-1}(\D \otimes I_{m})\eta\}\\
&\qquad\quad\qquad+\xi^T (\D^T \otimes I_{m} )\Phi\\
&\leq -\beta \{\xi^T (\D^T \otimes I_{m} )[\nabla^2\tilde{L}(x,t)]^{-1}(\D \otimes I_{m})\xi\}\\
&\qquad\quad\qquad+\xi^T (\D^T \otimes I_{m} )\Phi\\
&\leq -\beta\lambda_{\min}[(\nabla^2\tilde{L})^{-1}]\|(\D\otimes I_m)\xi\|_2^2+2\bar{\phi}mn^2|\E|,
\end{aligned}\end{equation}
If there exists an edge $(i_2, j_2)\in\E$ such that $x_{i_2}\neq x_{j_2}$, then $\|(\D\otimes I_m)\xi\| \geq 1$. It follows that
\begin{equation} \tilde{a}\leq-\beta\lambda_{\min}[(\nabla^2\tilde{L})^{-1}]+2\bar{\phi}mn^2|\E|.
\end{equation}
Since $\beta> \frac{2\bar{\phi}mn^2|\E|}{\min\{\lambda_{\min}[(\nabla^2\tilde{L}_i)^{-1}]\}}=\frac{2\bar{\phi}mn^2|\E|}{\lambda_{\min}[(\nabla^2\tilde{L})^{-1}]}$, it follows that if there exists an edge $(i_2, j_2)\in\E$ such that $x_{i_2}\neq x_{j_2}$, then $\tilde{a}< 0$. It is clear that $W_2(t)$ converges to zero eventually. That is, all agents reach a consensus eventually, i.e., $\lim\limits_{t\to\infty}\|x_i(t)-\sum\limits_{j=1}^nx_j(t)\|_2=0$ for all $i\in\V$. 

\section{Proof of Lemma \ref{lemma:phi:ineq}}\label{appendix:phi}
To begin with, we prove that each $x_i$ associated with the system \eqref{system} under controller \eqref{algorithm} remains in a bounded region, which in turn guarantees that all $\phi_i$ are bounded. Note that \cref{assumption:costfunction,assumption:costfunction:ineq,assumption:slater}, the initial condition \eqref{ini:x}, and the gain condition \eqref{gain:rho} hold. Then using a similar analysis to that in Lemma \ref{lemma:sumgradient}, we have each $\tilde{L}_i(x_i,t)$ is continuously differentiable and strongly convex in $x_i$. The derivative of $\nabla \tilde{L}_i(x_i,t)$ is shown in \eqref{dotgradient}. Assume that there exists at least one $x_i$ such that $x_i=+\infty\ \text{or} -\infty$. Then due to the strongly convexity and the continuously differentiability of $\tilde{L}_i(x_i,t)$, it is easy to show that $\nabla\tilde{L}_i(+\infty,t)=+\infty$. Note that \cref{assumption:graph,assumption:costfunction,assumption:costfunction:ineq,assumption:slater}, the initial condition \eqref{ini:x}, and the gain condition \eqref{gain:rho} hold. Then it follows from Lemma \ref{lemma:sumgradient} that it is impossible that all $x_i$ go to infinity at the same time. Without loss of generality, let us assume that $x_i=+\infty=\max\limits_{j\in\V}x_j$. It follows that $-\beta\sum\limits_{j\in\N_i}\SGN(x_i-x_j)\leq 0$ when $x_i=+\infty$. Therefore, from \eqref{dotgradient}, it is clear that $\dot{\nabla} \tilde{L}_i(x_i,t)$ must be negative when $x_i=+\infty$. Similarly, $\dot{\nabla} \tilde{L}_i(x_i,t)$ must be positive when $x_i=-\infty$. The decreasing $\nabla \tilde{L}_i(x_i,t)$ when $x_i=+\infty$ and increasing $\nabla \tilde{L}_i(x_i,t)$ when $x_i=-\infty$ will result in a bounded $\nabla\tilde{L}_i(x_i,t)$ and a bounded $x_i$, which contradicts with the unbounded $x_i$ assumption. Hence all $x_i$ must be bounded. 

Then, we will prove that all $\nabla\tilde{L}_i(x_i,t)$ are bounded for all time. It follows from Lemma \ref{lemma:sumgradient} that $\sum\limits_{i=1}^n\nabla \tilde{L}_i(x_i,t)$ is always bounded. Since all $x_i$ are bounded, we have all $\nabla f_i(x_i,t)$ and $\nabla g_{ij}(x_i,t)$ must be bounded. Then using an argument similar to Lemma 2 in \cite{Fazlyab}, all $ \frac{1}{1-\rho(t)g_{ij}(x_i,t)}$ are bounded. Therefore each $\nabla\tilde{L}_i(x_i,t)$ is always bounded for all time and for all $i\in\V$.
	
Next, we will prove that all $[\nabla^2\tilde{L}_i(x_i,t)]^{-1}$ are bounded for all time. Since all $\tilde{L}_i(x_i,t)$ are continuous differentiable and strongly convex in its corresponding $x_i$, then based on the statements in Section 9.1.2 in \cite{Boyd}, we know that all $\nabla^2\tilde{L}_i(x_i,t)$ satisfy $$m(t) I_n\leq\nabla^2\tilde{L}_i(x_i,t)\leq M(t) I_n,$$
with $m(t), M(t) \in\R_{>0}$, which implies that all $[\nabla^2\tilde{L}_i(x_i,t)]^{-1}$ are bounded and positive definite for all time. 
	
At last, given that all $\nabla\tilde{L}_i(x_i,t)$ and $\nabla ^2 \tilde{L}_i(x_i,t)$ are bounded for all time, under \cref{assumption:bound,assumption:bound:ineq,assumption:bound}, it is easy to see that all $\frac{\partial}{\partial t}\nabla\tilde{L}_i(x_i,t)$ remain bounded for all time.
	
Since $[\nabla^2\tilde{L}_i(x_i,t)]^{-1}$, $\nabla\tilde{L}_i(x_i,t)$ and $\frac{\partial}{\partial t}\nabla\tilde{L}_i(x_i,t)$ are bounded for all $i\in\V$ and for all $t\geq 0$, we can get the conclusion that $\phi_i(t)$ is bounded for all $i\in\V$ and for all $t\geq0$.

\end{appendices}
\end{document}